\documentclass{article}

\usepackage{amssymb}
\usepackage{amsthm}
\usepackage{amsmath}
\usepackage{array,hhline}

\newtheorem{lemma}{Lemma}[section]
\newtheorem{theorem}[lemma]{Theorem}
\newtheorem{proposition}[lemma]{Proposition}
\newtheorem{conjecture}[lemma]{Conjecture}
\newtheorem{corollary}[lemma]{Corollary}
\theoremstyle{definition}
\newtheorem{definition}[lemma]{Definition}
\newtheorem{remark}[lemma]{Remark}

\numberwithin{equation}{section}
\numberwithin{figure}{section}

\newcommand{\Xset}{\mathcal{X}}

\newcommand{\eg}{\emph{e.g.}}

\newcommand{\upperRomannumeral}[1]{\uppercase\expandafter{\romannumeral#1}}

\begin{document}

\author{\centerline {\large S. Toroyan} \\ {Yerevan State University, Yerevan, Armenia}}

\title{{\huge On a conjecture in bivariate interpolation}}

\maketitle

\begin{abstract}
\noindent The Gasca-Maeztu conjecture for the case
n=4 was proved for the first time in [J.~R.~Busch,
{A note on Lagrange interpolation in  $\mathbb{R}^2$},
Rev.\ Un.\ Mat.\ Argentina \ {\bf 36}  (1990)  33--38].
Here we bring a short and simple proof of it.
\end{abstract}

\section{Introduction\label{sec:intro}}
Denote by $\Pi_n$ the space of bivariate polynomials of total degree
not greater than $n$. We have
\begin{equation*}
N:=\dim \Pi_n=\binom{n+2}{2}.
\end{equation*}
We call a set $\Xset_s=\{ (x_1, y_1), (x_2, y_2), \dots , (x_s, y_s) \}$ of
distinct nodes $n$-poised, if for any data $\{c_1, \dots, c_s\}$ there exists
a unique polynomial $p\in \Pi_n$ satisfying the conditions
\begin{equation}\label{int cond}
p(x_i, y_i) = c_i, \ \ \quad i = 1, 2, \dots s .
\end{equation}

A necessary condition of
$n$-poisedness is: $s = N.$

A polynomial $p \in \Pi_n$ is called  an $n$-fundamental polynomial
for a node $ A = (x_k, y_k) \in \Xset_s$ if
\begin{equation*}
 p(x_k, y_k)=1 \quad \text{and} \quad p(x_i, y_i)=0, \quad \text{for all} \quad  1 \leq i \leq  s , \,  i \not= k.
\end{equation*}
 We denote the
$n$-fundamental polynomial of $A \in\Xset_s$ by $p_A^\star.$

We shall use the same letter, say $\ell$, to denote the
polynomial $\ell\in\Pi_1$ and the line with an equation
$\ell(x, y)=0.$

Now consider an $n$-poised set $\Xset = \Xset_{N} $. We say, that a node $A\in\Xset$
uses a line $\ell$, if $\ell$ is a factor of the fundamental
polynomial $p_A^\star$. We are going to use the following well known
\begin{proposition}[Bezout]\label{pointsells}
Suppose that  $\ell$ is a line and
$\Xset$ is a set of $n+1$ nodes lying in $\ell$.
Then for any polynomial $p \in \Pi_n$ vanishing on $\Xset$ we have
\begin{equation*}
p = \ell r  ,\quad \text{where} \quad r\in\Pi_{n-1}.
\end{equation*}
\end{proposition}
 It follows from
the Proposition~\ref{pointsells} that at most $n+1$ $n$-independent nodes can lye in a line.

We will make use of a special case of Cayley-Bacharach Theorem
(see \eg, \cite{E96},Th.~CB4, \cite{HJZ09a}, Prop.~4.1):
\begin{theorem}\label{CB}
Assume that the three lines  $\ell_1, \ell_2, \ell_3$ intersect another three lines
$\ell^{'}_1, \ell^{'}_2, \ell^{'}_3$ at  nine  different points. If a
polynomial $p\in \Pi_{3}$ vanishes at any eight  intersection
points, then it vanishes at all nine points.
\end{theorem}

\vspace{.5cm}
\section{The Gasca-Maeztu conjecture}
\vspace{.5cm}

Here we  consider a special type of $n$-poised sets.
\begin{definition}
We  call an  n-poised  set  $\Xset$ $GC_n$-set if  each node
$A\in\Xset$ has an $n$-fundamental polynomial which is a product of
$n$ linear factors.
\end{definition}
 Since the fundamental polynomial of an $n$-poised
 set is unique each of these lines passes through at least two nodes
 from $\Xset$, not belonging to the other lines
(see e.g. \cite{HJZ09b}, Lemma 2.5).

Next we bring the Gasca-Maeztu conjecture:
\begin{conjecture}[Gasca, Maeztu \cite{GM82}]
Any  $GC_n$-set contains $n+1$ collinear nodes.
\end{conjecture}
So far the conjecture was proved for the values $n\le 5$ (see \cite{HJZ2}).
In the case $n=4$ this reduces to the following:

\begin{theorem}\label{thmain}
Any  $GC_4$-set $\Xset$ of 15 nodes contains  five collinear nodes.
\end{theorem}
To prove this, we shall assume from now on:
\begin{equation}\label{assumption}
\textit{The set $\Xset$ is a
$GC_4$-set which does not contain  five collinear nodes,}
\end{equation}
in order to derive a contradiction.

For each node $A \in \Xset$ the 4-fundamental polynomial is a
product of four linear factors. In view of assumption~\eqref{assumption}
 the $14$  nodes of $\Xset \setminus \{A\}$ are distributed
in the four lines used by $A$ in two possible ways: $4+4+4+2$ or
$4+4+3+3$. Accordingly, we can represent $p^\star_A$ in two forms:
\begin{equation}\label{form1}
p^\star_A = \ell_{=4} \ell^{'}_{=4} \ell^{''}_{=4} \ell_{\ge 2},
\quad \quad p^\star_A = \ell_{=4} \ell^{'}_{=4} \ell_{\ge 3}
\ell^{'}_{\ge 3}.
\end{equation}
The lines with ${=k}$ in the subscript
are called $k$-node lines and pass through exactly $k$ nodes.
 The lines with ${\ge k}$ in the subscript pass
through $k$ nodes and  possibly also through some other, already
counted nodes, which are the intersection points with the other
lines.

\vspace{.5cm}
\subsection{Lines used by  several points}
 \vspace{.5cm}
We start with two lemmas from \cite{GC4simple} (see Lemma 2.5 and
Lemma 2.6 there). For the sake of completeness we bring here the
proofs.
\begin{lemma}\label {lm23}
Any $2$ or $3$-node line can be used by at most one node of $\Xset.$
\end{lemma}

\begin{proof}
Assume by way of contradiction that $\tilde{\ell}$ is a $2$ or $3$-node
line used by two points $A, B \in \Xset$. Consider the fundamental
polynomial $p^{\star}_A$ of $A$. It consists of the line
$\tilde{\ell}$ and three more lines, which contain the remaining
$\ge 11$ nodes of $\Xset \setminus (\tilde\ell\cup \{A\}),$
including $B$. Since there is no $5$-node line, we get
\begin{equation*}\label{form8}
p^\star_A = \tilde{\ell} \ell_{=4} \ell^{'}_{=4} \ell_{\ge 3}.
\end{equation*}
First suppose that $B$ belongs to one of the $4$-node lines, say to
$\ell^{'}_{=4}.$  We have that
\begin{equation*}
p^\star_B=\tilde{\ell}  q, \ \text{where}\ \  q \in \Pi_3.
\end{equation*} Notice that $q$
vanishes at $4$ nodes of $\ell_{=4}$, therefore according to
Proposition\ref{pointsells}
\begin{equation*}
p^\star_B=\tilde{\ell} \ell_{=4}  r, \ \text{where}\ \  r \in \Pi_2.
\end{equation*}
Now $r$ vanishes at $3$ nodes of $\ell^{'}_{=4}$ (i.e., except $B$).
Therefore again from Proposition~\ref{pointsells} we get that $r$ vanishes
at all points of $\ell^{'}_{=4}$ including $B.$ Hence $p^\star_B$
vanishes at $B,$ which is a contradiction.

Now assume that $B$ belongs to the line $\ell_{\ge 3}.$ Then $q$
vanishes at $4$ nodes of $\ell_{=4},$ $4\ (\ge 3)$ nodes of
$\ell^{'}_{=4}$  and at least $2$ nodes of $\ell_{\ge 3}.$ Therefore
we get from Proposition~\ref{pointsells} as above
\begin{equation*}
p_{B}^{\star}= \ell_{=4} \ell^{'}_{=4} \ell_{\geq 3} \ell \, ,\quad \text{where} \quad \ell \in \Pi_1 .
\end{equation*}
Hence again $p^\star_B$ vanishes at $B,$ which
is a contradiction.
\end{proof}
\begin{lemma}\label{lm4}
Any $4$-node line $\tilde{\ell}$ can be used by at most three nodes of $\Xset.$
 If three nodes use $\tilde{\ell}$  then they share two more $4$-node lines.

\end{lemma}
\begin{proof}
Assume that $\tilde{\ell}=\tilde{\ell}_{=4}$ is a $4$-node line and
is used by two nodes $A, B \in \Xset$. Consider the fundamental
polynomial $p^\star_A$ of $A \in \Xset$. It consists of the line
$\tilde{\ell}$ and three more used lines containing the remaining
$10$ nodes  of $\Xset \setminus (\ell\cup \{A\}),$ including $B$.
Now there are two possibilities:
\begin{equation}\label{form3}
p^\star_A = \tilde{\ell} \ell_{=4} \ell^{'}_{=4} \ell_{\ge 2},
\end{equation}
\begin{equation}\label{form4}
p^\star_A = \tilde{\ell} \ell_{=4} \ell_{\ge 3} \ell^{'}_{\ge 3}.
\end{equation}

Consider first the case \eqref{form3}. We readily get, as in the
proof of Lemma \ref{lm23}, that $B$ does not belong to any of the
$4$-node lines. Hence $B$ must belong to the line $\ell_{\ge 2}.$
The same statement is true also for the third point $C$ that may use
$\tilde{\ell}.$ Clearly there is no room for the fourth node in
$\ell_{\geq{2}}.$ Thus Lemma is proved in this case. Notice that
in this case $p^\star_B$ ($p^\star_A$) uses the lines
$\tilde{\ell}, \ell_{=4}, \ell^{'}_{=4}$ and the line
 $\ell_{AC}$ ($\ell_{AB}$),
where $\ell_{AC}$ is the line passing through $A$ and $C$.
Therefore the nodes $A, B, C$ share three 4-node lines.

Now consider the case when $p^\star_A$ is given by \eqref{form4}. We
readily get as above that $B$ does not belong to any of the $3$-node
lines. Hence $B$ must belong to the line $\ell_{=4}.$ We have also
\begin{equation}\label{form7}
p^\star_B = \tilde{\ell} \alpha_{=4} \alpha_{\ge 3} \alpha^{'}_{\ge
3}.
\end{equation}
Moreover we have that the node $A$ in its turn belongs to the line
$\alpha_{=4}.$

Now let us verify that the nodes $A$ and $B$ do not share any line
except $\tilde{\ell}.$ Indeed, $\ell_{=4}$ and $\alpha_{=4}$ are the
only lines containing $B$ and $A,$ respectively. Thus they are not
among the possibly coinciding lines. Without loss of generality
assume by way of contradiction that $\ell^{'}_{\ge 3}\equiv \alpha^{'}_{\ge
3}.$

Thus we have
\begin{equation*}
p^\star_B=\tilde{\ell} \ell^{'}_{\ge 3} q, \ \text{where}\ \  q \in
\Pi_2.
\end{equation*}
Then, in view of \eqref{form4}, $q$ vanishes at $3$ nodes of
$\ell_{\ge 3}$ and at least at $2$ nodes of $\ell_{=4}$ (i.e., except
 $B$ and a possible secondary node of $\ell^{'}_{\ge 3}$). Therefore from
Proposition \ref{pointsells} we get that $q$ vanishes at all points of
the lines $\ell_{=4}$ and $\ell_{\ge 3}$ including $B.$ Hence
$p^\star_B$ vanishes at $B,$ which is a contradiction.

Thus the triples of the used lines $\ell_{=4}, \ell^{'}_{\ge 3},
\ell_{\ge 3}$ and $\alpha_{=4}, \alpha^{'}_{\ge 3}, \alpha_{\ge 3}$
intersect at exactly $9$ nodes of ${\mathcal I}:= \Xset\setminus
[\tilde\ell \cup \{A,B\}].$

If a third node $C$ is using $\tilde{\ell}$ then we have that $C\in
{\mathcal I}$ and $p_C^*$ must vanish at eight nodes of ${\mathcal I}$
but by the Theorem~\ref{CB} $p_C^*$ vanishes at $C$ also, which
is a contradiction. Therefore the 4-node line, in this case, can be
used at most twice.
\end{proof}

\begin{corollary} \label{cor} Suppose $\ell$ is a $4$-node line and three nodes $A,B,C\in\Xset\setminus \ell$ use a line
$\tilde{\ell},$ Then $l$ is among the three lines used by $A,B$ and
$C.$
\end{corollary}
Indeed, the $12$ nodes of $\Xset\setminus\{A,B,C\}$ lie in the three
used lines.

 With the following Lemma we strengthen the Lemma 2 in \cite{Bush}.
\begin{lemma}\label{3node}
If a node $A \in \Xset$ uses a $4$-node line ${\ell}$ then there are
two more nodes in $\Xset$ using it.
\end{lemma}

\begin{proof}
Denote the four lines joining the node $A$ with the four nodes on
${\ell}$  by $\ell_1, \ell_2, \ell_3, \ell_4$. Consider any node
from $\Xset \setminus ( \ell  \cup \{A\}).$ The four lines
used by it pass through five nodes in $\ell\cup \{A\}$.
Therefore one of lines passes through $2$ nodes of the five and
therefore coincides with one of the lines $\ell_1, \ell_2, \ell_3,
 \ell_4$ or ${\ell}$.
Thus each of eleven nodes of $\Xset \setminus {\ell}$ uses one of
the lines ${\ell}, \ell_1, \ell_2, \ell_3, \ell_4$. Hence there are
at least three nodes in $ \Xset \setminus {\ell}$ using the
same line $\tilde\ell$ from the mentioned $5.$ In view of Corollary
\ref{cor} $\tilde{\ell}$ is used by the same triple of nodes. In view of
Lemma \ref{lm4} the node $A$ is among the three nodes.
\end{proof}

\begin{remark}\label{rem} Notice, that $\tilde\ell\equiv\ell.$ Indeed $\tilde\ell$ is
used by $A$ and therefore it cannot coincide with any $\ell_i,
i=1,2,3,4$, Also, each of the lines ${\ell}, \ell_1, \ell_2, \ell_3,
\ell_4$ is used by exactly $2$ nodes from $10$ of $\Xset \setminus
({\ell} \cup \{A\}).$ Therefore in view of Lemma \ref{lm4} each of the lines $\ell_1, \ell_2,
\ell_3, \ell_4$ is a $4$-node line.
\end{remark}

\subsection{Proof of the conjecture}

It follows from Lemmas \ref{lm4} and \ref{3node} that all the fundamental polynomials of
$\Xset$ have the form  $4+4+4+2,$ i.e., the first case of
 \eqref{form1}.

Suppose a node $A$ uses a  4-node line ${\ell}$ and the four lines
$\ell_1, \ell_2, \ell_3, \ell_4$ pass through $A$ and the four nodes
in ${\ell}$, as in the proof of Lemma \ref{3node}. The line
 ${\ell}$ is used by two more nodes $B, C \notin
\ell.$  (Lemma~\ref{3node}). The nodes $A,B,C$ share two
more $4$-node lines which we denote by $\ell^{'}$ and $\ell^{''}.$ Let us
verify that the nodes $B$ and $C$ do not lie in the lines $\ell_1,
\ell_2, \ell_3, \ell_4.$ Indeed, suppose conversely, say, the node $B$ is in
$\ell_1.$ Then, in view of Remark \ref{rem}, $C$ does not use any of
the $4$-node lines $\ell_1, \ell_2, \ell_3, \ell_4$ while from the
proof of Lemma \ref{lm4} we have that $C$ uses the line passing
through $A$ and $B$ and thus coinciding with $\ell_1.$

Therefore, $12$ nodes of $\Xset\setminus \{A,B,C\}$ belong to the
lines $\ell_1, \ell_2, \ell_3, \ell_4.$ From the other side these
$12$ nodes belong to the lines $\ell, \ell^{'}$ and $\ell^{''}.$

Now, we may conclude that $12$ nodes of $\Xset\setminus \{A,B,C\}$
are the intersection points of the $4$ lines $\ell_1, \ell_2,
\ell_3, \ell_4.$ with the $3$ lines $\ell, \ell^{'}$ and $\ell^{''}.$

Finally consider  the polynomial $p=\ell_1 \ell_2 \ell_3 \ell_4.$ As the
fourth degree polynomial $p$ vanishes at all the nodes but $B$ and
$C$, it should be a linear combination of the fundamental
polynomials of these two nodes. Both these fundamental polynomials
vanish on the lines $\ell, \ell^{'}, \ell^{''}$,
so this should be true also for p, which is a contradiction. The
Theorem~\ref{thmain} is proved.


\end{document}